\documentclass[11pt,a4paper,reqno,intlimits,sumlimits]{amsart}

\usepackage{amssymb}
\usepackage{chicago}
\usepackage{color}
\usepackage[mathscr]{euscript}
\usepackage{xspace}
\usepackage{epsfig,rotating}
\usepackage{enumerate}

\begin{document}

\frenchspacing

%% ================================================================ %%
%% ================================================================ %%
\theoremstyle{plain}
\newtheorem{theorem}{Theorem}[section]
\newtheorem{lemma}[theorem]{Lemma}
\newtheorem{proposition}[theorem]{Proposition}
\newtheorem{corollary}[theorem]{Corollary}

\theoremstyle{definition}
\newtheorem{remark}[theorem]{Remark}
\newtheorem{definition}[theorem]{Definition}
\newtheorem{assumption}{Assumption}
\newtheorem*{assuL}{Assumption ($\mathbb{L}$)}
\newtheorem*{assuAC}{Assumption ($\mathbb{AC}$)}
\newtheorem*{assuEM}{Assumption ($\mathbb{EM}$)}
%% ================================================================ %%
\renewcommand{\theequation}{\thesection.\arabic{equation}}
\numberwithin{equation}{section}

\renewcommand{\thetable}{\thesection.\arabic{table}}
\numberwithin{table}{section}

\renewcommand{\thefigure}{\thesection.\arabic{figure}}
\numberwithin{figure}{section}

\newcommand{\Law}{\ensuremath{\mathop{\mathrm{Law}}}}
\newcommand{\loc}{{\mathrm{loc}}}
\newcommand{\Log}{\ensuremath{\mathop{\mathcal{L}\mathrm{og}}}}

\let\SETMINUS\setminus
\renewcommand{\setminus}{\backslash}

\def\stackrelboth#1#2#3{\mathrel{\mathop{#2}\limits^{#1}_{#3}}}

\newcommand{\prozess}[1][L]{{\ensuremath{#1=(#1_t)_{0\le t\le T_*}}}\xspace}
\newcommand{\prazess}[1][L]{{\ensuremath{#1=(#1_t)_{0\le t\le T_*}}}\xspace}
\newcommand{\pt}[1][]{{\ensuremath{\P_{T_{#1}}}}\xspace}
\newcommand{\ts}[1][]{\ensuremath{T_{#1}}\xspace}
%% ================================================================ %%
\def\P{\ensuremath{\mathrm{I\kern-.2em P}}}
\def\E{\mathrm{I\kern-.2em E}}

\def\bF{\mathbf{F}}
\def\F{\ensuremath{\mathcal{F}}}
\def\R{\ensuremath{\mathbb{R}}}
\def\C{\ensuremath{\mathbb{C}}}
\def\bt{\ensuremath{\mathbf{T}}}

\def\Rmz{\R\setminus\{0\}}
\def\Rdmz{\R^d\setminus\{0\}}
\def\Rnmz{\R^n\setminus\{0\}}

\def\Rp{\mathbb{R}_{\geqslant0}}

\def\lev{L\'{e}vy\xspace}
\def\lib{LIBOR\xspace}
\def\lk{L\'{e}vy--Khintchine\xspace}
\def\smmg{semimartingale\xspace}
\def\smmgs{semimartingales\xspace}
\def\mg{martingale\xspace}
\def\tih{time-inhomoge\-neous\xspace}

\def\eqlaw{\ensuremath{\stackrel{\mathrrefersm{d}}{=}}}

\def\ud{\ensuremath{\mathrm{d}}}
\def\dt{\ud t}
\def\ds{\ud s}
\def\dx{\ud x}
\def\dy{\ud y}
\def\dsdx{\ensuremath{(\ud s, \ud x)}}
\def\dtdx{\ensuremath{(\ud t, \ud x)}}

\def\intrr{\ensuremath{\int_{\R}}}

\def\MM{\ensuremath{\mathscr{M}}}
\def\ME{\mathbb{M}}

\def\EM{\ensuremath{(\mathbb{EM})}\xspace}
\def\ES{\ensuremath{(\mathbb{ES})}\xspace}
\def\AC{\ensuremath{(\mathbb{AC})}\xspace}
\def\LL{\ensuremath{(\mathbb{L})}\xspace}

\def\ott{{0\leq t\leq T_*}}

\def\e{\mathrm{e}}
\def\eps{\epsilon}

\def\half{\frac{1}{2}}
\def\LibT{L(t,T_i)}
\def\MeaT{\P_{T_{i+1}}}
\def\volT{\lambda(s,T_i)}
\def\vol2T{\lambda^2(s,T_i)}
\def\LevT{H_s^{T_{i+1}}}
%% ================================================================ %%
%% ================================================================ %%

\title[Taylor Approximation of SDE and L\'evy LIBOR model]
      {Strong Taylor approximation of Stochastic Differential Equations
       and application to the L\'evy LIBOR model}

\author[A. Papapantoleon]{Antonis Papapantoleon}
\author[M. Siopacha]{Maria Siopacha}

\address{Institute of Mathematics, TU Berlin, Stra\ss e des 17. Juni 136,
         10623 Berlin, Germany \& Quantitative Products Laboratory,
         Deutsche Bank AG, Alexanderstr. 5, 10178 Berlin, Germany}
\email{papapan@math.tu-berlin.de}

\address{Raiffeisen Zentralbank \"Osterreich AG, Market Risk,
         Am Stadtpark 9, 1030 Vienna, Austria}
\email{maria.siopacha@rzb.at}


\thanks{We are grateful to Friedrich Hubalek for numerous discussions and
        assistance with computational issues.
        We also thank Ernst Eberlein, Josef Teichmann, Christoph Schwab and
        Christoph Winter for helpful discussions and comments.
        A.~P. gratefully acknowledges the financial support from the Austrian
        Science Fund (FWF grant Y328, START Prize)}

\date{}
\maketitle
\pagestyle{myheadings}

\begin{abstract}
In this article we develop a method for the strong approximation of stochastic
differential equations (SDEs) driven by L\'evy processes or general semimartingales.
The main ingredients of our method is the perturbation of the SDE and the Taylor
expansion of the resulting parameterized curve. We apply this method to develop
strong approximation schemes for LIBOR market models. In particular, we derive
fast and precise algorithms for the valuation of derivatives in LIBOR models
which are more tractable than the simulation of the full SDE. A numerical
example for the L\'evy LIBOR model illustrates our method.
\end{abstract}

\section{Introduction}
\label{intro}

The main aim of this paper is to develop a general method for the strong
approximation of stochastic differential equations (SDEs) and to apply it to the
valuation of options in \lib models. The method is based on the perturbation of
the initial SDE by a real parameter, and then on the Taylor expansion of
the resulting parameterized curve around zero. Thus, we follow the line of
thought of \citeN{SiopachaTeichmann07} and extend their results from continuous
to general \smmgs. The motivation for this work comes from \lib market models;
in particular, we consider the \lev \lib model as the basic paradigm for the
development of this method.

The \lib market model has become a standard model for the pricing of interest
rate derivatives in recent years. The main advantage of the LIBOR model in
comparison to other approaches is that the evolution of discretely compounded,
market-observable forward rates is modeled directly and not deduced from the
evolution of unobservable factors. Moreover, the log-normal \lib model is
consistent with the market practice of pricing caps according to Black's formula
(cf. \citeNP{Black76}). However, despite its apparent popularity, the \lib
market model has certain well-known pitfalls.

On the one hand, the log-normal \lib model is driven by a
Brownian motion, hence it cannot be calibrated adequately to the observed market
data. An interest rate model is typically calibrated to the implied volatility
surface from the cap market and the correlation structure of at-the-money
swaptions. Several extensions of the \lib model have been proposed in the
literature using jump-diffusions, \lev processes or general semimartingales as
the driving motion (cf. e.g. \citeNP{GlassermanKou03}, Eberlein and \"Ozkan
\citeyear{EberleinOezkan05}, \citeNP{Jamshidian99}), or incorporating
stochastic volatility effects (cf. e.g. \citeNP{AndersenBrothertonRatcliffe05}).

On the other hand, the dynamics of LIBOR rates are not tractable under
every forward measure due to the random terms that enter the dynamics
of \lib rates during the construction of the model. In particular, when the
driving process has continuous paths the dynamics of LIBOR rates are tractable
under their corresponding forward measure, but they are not tractable under any
other forward measure. When the driving process is a general semimartingale,
then the dynamics of \lib rates are not even tractable under their very own
forward measure. Consequently:
\begin{enumerate}
\item if the driving process is a \textit{continuous} \smmg caplets can be
      priced in closed form, but \textit{not} swaptions or other multi-LIBOR
      derivatives;
\item if the driving process  is a \textit{general} \smmg, then even caplets
      \textit{cannot} be priced in closed form.
\end{enumerate}
The standard remedy to this problem is the so-called ``frozen drift''
approximation, where one replaces the random terms in the dynamics of \lib rates
by their deterministic initial values; it was first proposed by
\shortciteN{BraceGatarekMusiela97} for the pricing of swaptions and has been
used by several authors ever since. \shortciteN{BraceDunBarton01},
\shortciteN{DunBartonSchloegl01} and \citeN{Schloegl02} argue that freezing the
drift is justified, since the deviation from the original equation is small in
several measures.

Although the frozen drift approximation is the simplest and most popular solution,
it is well-known that it does not yield acceptable results, especially
for exotic derivatives and longer horizons. Therefore, several other
approximations have been developed in the literature; in one line of research
\citeN{DanilukGatarek05} and \shortciteN{KurbanmuradovSabelfeldSchoenmakers} are
looking for the best lognormal approximation of the forward LIBOR dynamics; cf.
also \citeN{Schoenmakers05}. Other authors have been using linear interpolations
and predictor-corrector Monte Carlo methods to get a more accurate approximation
of the drift term (cf. e.g. \shortciteNP{PelsserPieterszRegenmortel},
\shortciteANP{HunterJaeckelJoshi01} \citeyearNP{HunterJaeckelJoshi01}
and \citeNP{GlassermanZhao00}). We refer the reader to Joshi and Stacey
\citeyear{JoshiStacey08} for a detailed overview of that literature, and
for some new approximation schemes and numerical experiments.

Although most of this literature focuses on the lognormal LIBOR market model,
\citeANP{GlassermanMerener03} (\citeyearNP{GlassermanMerener03},
\citeyearNP{GlassermanMerener03b}) have developed approximation schemes for the
pricing of caps and swaptions in jump-diffusion \lib market models.

In this article we develop a general method for the approximation of the random
terms that enter into the drift of LIBOR models. In particular, by perturbing
the SDE for the \lib rates and applying Taylor's theorem we develop a generic
approximation scheme; we concentrate here on the first order Taylor expansion,
although higher order expansions can be considered in the same framework. At the
same time, the frozen drift approximation can be embedded in this method as the
zero-order Taylor expansion, thus offering a theoretical justification for this
approximation. The method we develop yields more accurate results than the
frozen drift approximation, while being computationally simpler than the
simulation of the full SDE for the \lib rates. Moreover, our method is universal
and can be applied to any \lib model driven by a general \smmg. However, we focus
on the \lev \lib model as a characteristic example of a \lib model driven by a
general semimartingale.

The article is structured as follows: in section \ref{PIIAC} we review \tih
\lev process, and in section \ref{LevyLIBOR} we revisit the \lev \lib model. In
section \ref{dynamics} we describe the dynamics of log-\lib rates under the terminal
martingale measure and express them as a \lev-driven SDE. In section \ref{main}
we develop the strong Taylor approximation method and apply it to the \lev \lib
model. Finally, section \ref{numerics} contains a numerical illustration.

\section{L\'evy processes}
\label{PIIAC}

Let ($\Omega, \F, \bF, \P$) be a complete stochastic basis, where $\F=\F_{T_*}$
and the filtration $\bF=(\F_t)_{t\in[0,T_*]}$ satisfies the usual conditions;
we assume that $T_*\in\Rp$ is a finite time horizon. The driving process
\prozess[H] is a \emph{process} with \emph{independent increments} and
\emph{absolutely continuous} characteristics; this is also called a \emph{\tih
\lev process}. That is, $H$ is an adapted, c\`{a}dl\`{a}g, real-valued stochastic
process with independent increments, starting from zero, where the law of $H_t$,
$t\in[0,T_*]$, is described by the characteristic function
\begin{align}\label{LK}
\E\!\left[\e^{iuH_{t}}\right]
 = \exp\left(\int_{0}^{t}\Big[ ib_su - \frac{c_s}{2}u^{2}
  + \int_{\R}(\e^{iux}-1-iux)F_s(\ud x)\Big]\ud s\right);
\end{align}
here $b_t\in\R$, $c_t\in\Rp$ and $F_t$ is a \lev measure, i.e. satisfies $F_t(\{0\})=0$
and $\int_{\R}(1\wedge|x|^2)F_t(\ud x)<\infty$, for all $t\in[0,T_*]$. In addition,
the process $H$ satisfies Assumptions ($\mathbb{AC}$) and ($\mathbb{EM}$) given
below.

\begin{assuAC}
The triplets ($b_t,c_t,F_t$) satisfy
\begin{eqnarray}
\int_{0}^{T_*}\bigg( |b_t| + c_t
  + \int_{\R}(1\wedge|x|^2)F_t(\ud x) \bigg)\ud t <\infty.
\end{eqnarray}
\end{assuAC}

\begin{assuEM}
There exist constants $M, \varepsilon>0$ such that for every
$u\in[-(1+\varepsilon)M,(1+\varepsilon)M]=:\ME$
\begin{equation}\label{eq:Int}
    \int_0^{T_*}\int_{\{|x|>1\}}\e^{ux} F_t(\ud x)\ud t<\infty.
\end{equation}
Moreover, without loss of generality, we assume that
$\int_{\{|x|>1\}}\e^{ux}F_t(\ud x)<\infty$ for all $t\in[0,T_*]$ and $u\in\ME$.
\end{assuEM}

These assumptions render the process \prozess[H] a \emph{special}
semimartingale, therefore it has the canonical decomposition (cf.
Jacod and Shiryaev
\citeyearNP[II.2.38]{JacodShiryaev03}, and \shortciteNP{EberleinJacodRaible05})
\begin{align}\label{canonical}
 H = \int_{0}^{\cdot}b_s\ud s
       + \int_{0}^{\cdot}\sqrt{c_s} \ud W_{s}
       + \int_{0}^{\cdot}\int_{\R} x(\mu^{H}-\nu)(\ud s,\ud x),
\end{align}
where $\mu^H$ is the random measure of jumps of the process $H$, $\nu$ is the
$\P$-compensator of $\mu^H$, and \prazess[W] is a $\P$-standard Brownian motion.
The \emph{triplet of predictable characteristics} of $H$ with respect to the
measure $\P$, $\mathbb T(H|\P)=(B,C,\nu)$, is
\begin{eqnarray}\label{ch4:triplet}
 B = \int_{0}^{\cdot}b_s\ud s, \qquad
 C = \int_0^\cdot c_s\ud s, \qquad
 \nu([0,\cdot] \times A)=\int_0^\cdot\int_A F_s(\ud x)\ud s,
\end{eqnarray}
where $A\in\mathcal{B}(\R)$; the triplet ($b,c,F$) represents the \emph{local
characteristics} of $H$. In addition, the triplet of predictable characteristics
($B,C,\nu$) determines the distribution of $H$, as the \lk formula \eqref{LK}
obviously dictates.

We denote by $\kappa_s$ the \emph{cumulant generating function} associated to
the infinitely divisible distribution with \lev triplet ($b_s,c_s,F_s$), i.e.
for $z\in\ME$ and $s\in[0,T_*]$
\begin{align}\label{cumulant}
\kappa_s(z) := b_sz+\frac{c_s}{2}z^{2}
             + \int_{\R}(\e^{zx}-1-zx)F_s(\ud x).
\end{align}
Using Assumption \EM we can extend $\kappa_s$ to the complex domain $\C$, for
$z\in\C$ with $\Re z\in\ME$, and the characteristic function of $H_t$ can be
written as
\begin{align}\label{char-fun}
\E\!\left[\e^{iuH_{t}}\right] = \exp\bigg(\int_{0}^{t} \kappa_s(iu)\ud s\bigg).
\end{align}
If $H$ is a \lev process, i.e. time-homogeneous, then ($b_s,c_s,F_s$) -- and thus
also $\kappa_s$ -- do not depend on $s$. In that case, $\kappa$ equals the cumulant
(log-moment) generating function of $H_1$.

\section{The L\'evy LIBOR model}
\label{LevyLIBOR}

The \lev LIBOR model was developed by \citeN{EberleinOezkan05}, following the
seminal articles of \shortciteN{SandmannSondermannMiltersen95}, 
\shortciteANP{MiltersenSandmannSondermann97}
\citeyear{MiltersenSandmannSondermann97} and
\shortciteN{BraceGatarekMusiela97} on LIBOR market models driven by Brownian
motion; see also \citeN{GlassermanKou03} and \citeN{Jamshidian99} for LIBOR
models driven by jump processes and general semimartingales respectively. The
\lev \lib model is a \textit{market model} where the forward LIBOR rate is
modeled directly, and is driven by a \tih \lev process.

Let $0=T_0<T_{1}<\cdots<T_{N}<T_{N+1}=T_*$ denote a discrete tenor
structure where $\delta_i=T_{i+1}-T_{i}$, $i\in\{0,1,\dots,N\}$. Consider a
complete stochastic basis $(\Omega, \F,\mathbf{F},\P_{T_*})$ and a \tih \lev
process \prozess[H] satisfying Assumptions $(\mathbb{AC})$ and $(\mathbb{EM})$.
The process $H$ has predictable characteristics ($0,C,\nu^{T_*}$) or local
characteristics $(0,c,F^{T_*})$, and its canonical decomposition is
\begin{align}\label{canon-LIBOR}
H = \int_0^\cdot \sqrt{c_s}\ud W_s^{T_*}
  + \int_0^\cdot\int_{\R}x(\mu^H-\nu^{T_*})\dsdx,
\end{align}
where $W^{T_*}$ is a $\P_{T_*}$-standard Brownian motion, $\mu^H$ is the random
measure associated with the jumps of $H$ and $\nu^{T_*}$ is the
$\P_{T_*}$-compensator of $\mu^H$. We further assume that the following
conditions are in force.
\begin{description}
\item[(LR1)] For any maturity $T_{i}$ there exists a bounded, continuous,
             deterministic function $\lambda(\cdot,T_{i}):[0,T_i]\rightarrow \R$,
             which represents the volatility of the forward LIBOR rate process
             $L(\cdot, T_{i})$. Moreover,
\begin{align*}
\sum_{i=1}^N \big|\lambda(s,T_i)\big|\leq M,
\end{align*}
             for all $s\in[0,T_*]$, where $M$ is the constant from Assumption
             ($\mathbb{EM}$), and $\lambda(s,T_i)=0$ for all $s>T_i$.
\item[(LR2)] The initial term structure $B(0,T_i)$, $1\leq i\leq N+1$, is strictly
             positive and strictly decreasing. Consequently, the initial term
             structure of forward LIBOR rates is given, for $1\leq i\leq N$, by
\begin{align}\label{i-val}
L(0,T_i)=\frac{1}{\delta_i}\left(\frac{B(0,T_i)}{B(0,T_i+\delta_i)}-1\right)>0.
\end{align}
\end{description}

The construction starts by postulating that the dynamics of the forward LIBOR
rate with the longest maturity $L(\cdot,T_N)$ is driven by the \tih \lev process
$H$ and evolve as a martingale under the terminal forward measure $\P_{T_*}$.
Then, the dynamics of the LIBOR rates for the preceding maturities are
constructed by backward induction; they are driven by the same process $H$ and
evolve as martingales under their associated forward measures.

Let us denote by $\MeaT$ the forward measure associated to the settlement date
$T_{i+1}$, $i\in\{0,\dots,N\}$. The dynamics of the forward LIBOR rate
$L(\cdot,T_i)$, for an arbitrary $T_i$, is given by
\begin{align}\label{LIBOR-dyn}
\LibT = L(0,T_i)
 \exp\left(\int_0^t b^L(s,T_i)\ud s+\int_0^t \volT\ud\LevT\right),
\end{align}
where $H^{T_{i+1}}$ is a special \textit{semimartingale} with canonical
decomposition
\begin{align}\label{LIBOR-Levy}
H_t^{T_{i+1}}
      = \int_0^t \sqrt{c_s}\ud W_s^{T_{i+1}}
      + \int_0^t\int_{\R}x(\mu^H-\nu^{T_{i+1}})\dsdx.
\end{align}
Here $W^{T_{i+1}}$ is a $\P_{T_{i+1}}$-standard Brownian motion and
$\nu^{T_{i+1}}$ is the $\P_{T_{i+1}}$-compensator of $\mu^H$. The dynamics of
an arbitrary LIBOR rate again evolves as a martingale under its corresponding
forward measure; therefore, we specify the drift term of the forward LIBOR
process $L(\cdot,T_i)$ as
\begin{align}\label{LIBOR-drift-term}
b^L(s,T_i)
 & = -\half \vol2T c_s\nonumber\\
 & \quad\;
     - \int_{\R} \big(\e^{\volT x}-1-\volT x\big)F_s^{T_{i+1}}(\ud x).
\end{align}

The forward measure $\MeaT$, which is defined on
$(\Omega,\F,(\F_t)_{0\leq t\leq T_{i+1}})$, is related to the terminal forward
measure $\P_{T_*}$ via
\begin{align}\label{TF-libor}
\frac{\ud\MeaT}{\ud\P_{T_*}}
 = \prod_{l=i+1}^{N}\frac{1+\delta_l L(T_{i+1},T_l)}{1+\delta_l L(0,T_l)}
 = \frac{B(0,T_*)}{B(0,T_{i+1})}\prod_{l=i+1}^{N} \left(1+\delta_l L(T_{i+1},T_l)\right).
\end{align}
The $\MeaT$-Brownian motion $W^{T_{i+1}}$ is related to the $\P_{T_*}$-Brownian
motion via
\begin{align}\label{Levy-LIBOR-Brownian}
W_t^{T_{i+1}}
 & = W_t^{T_{i+2}} - \int_0^t \alpha(s,T_{i+1})\sqrt{c_s}\ud s
   = \dots\nonumber\\
 & = W_t^{T_*} - \int_0^t \left(\sum_{l=i+1}^{N}\alpha(s,T_l)\right)\sqrt{c_s} \ud s,
\end{align}
where
\begin{align}\label{Levy-LIBOR-alpha}
\alpha(t,T_l)
 = \frac{\delta_l L(t-,T_l)}{1+\delta_l L(t-,T_l)}\lambda(t,T_l).
\end{align}
The $\MeaT$-compensator of $\mu^H$, $\nu^{T_{i+1}}$, is related to the
$\P_{T_*}$-compensator of $\mu^H$ via
\begin{align}\label{Levy-LIBOR-compensator}
\nu^{T_{i+1}}\dsdx
 &= \beta(s,x,T_{i+1})\nu^{T_{i+2}}\dsdx
  = \dots\nonumber\\
 &= \left(\prod_{l=i+1}^{N}\beta(s,x,T_l)\right)\nu^{T_*}\dsdx,
\end{align}
where
\begin{align}\label{Levy-LIBOR-beta}
\beta(t,x,T_l,)
 = \frac{\delta_l L(t-,T_l)}{1+\delta_l L(t-,T_l)}\Big(\e^{\lambda(t,T_l)x}-1\Big) +1.
\end{align}

\begin{remark}\label{LIBOR-conn}
Notice that the process $H^{T_{i+1}}$, driving the forward LIBOR rate $L(\cdot,T_i)$,
and $H=H^{T_*}$ have the same \emph{martingale} part and differ only in the
\emph{finite variation} part (drift). An application of Girsanov's theorem for
semimartingales yields that the $\P_{T_{i+1}}$-finite variation part of $H$ is
\begin{align*}
\int_0^{\cdot} c_s\sum_{l=i+1}^{N}\alpha(s,T_l) \ud s
  + \int_0^{\cdot}\intrr x\left(\prod_{l=i+1}^{N}\beta(s,x,T_l)-1\right)\nu^{T_*}\dsdx.
\end{align*}
\end{remark}

\begin{remark}\label{non-PIIAC}
The process $H=H^{T_*}$ driving the most distant LIBOR rate $L(\cdot,T_N)$ is --
by assumption -- a time-inhomogeneous \lev process. However, this is \emph{not}
the case for any of the processes $H^{T_{i+1}}$ driving the remaining LIBOR
rates, because the \emph{random} terms $\frac{\delta_lL(t-,T_l)}{1+\delta_lL(t-,T_l)}$
enter into the compensators $\nu^{T_{i+1}}$ during the construction; see
equations \eqref{Levy-LIBOR-compensator} and \eqref{Levy-LIBOR-beta}.
\end{remark}

\section{Terminal measure dynamics and log-\lib rates}
\label{dynamics}

In this section we derive the stochastic differential equation that the
dynamics of log-LIBOR rates satisfy under the terminal measure $\P_{T_*}$. This
will be the starting point for the approximation method that will be developed
in the next section. Of course, we could consider the
SDE as the defining point for the model, as is often the case in stochastic
volatility LIBOR models, cf. e.g. \citeN{AndersenBrothertonRatcliffe05}.

Starting with the dynamics of the LIBOR rate $L(\cdot,\ts[i])$ under the forward
martingale measure \pt[i+1], and using the connection between the forward and
terminal martingale measures (cf. eqs. \eqref{Levy-LIBOR-Brownian}--\eqref{Levy-LIBOR-beta}
and Remark \ref{LIBOR-conn}), we have that the dynamics of the LIBOR rate $L(\cdot,\ts[i])$
under the terminal measure are given by
\begin{align}\label{LIBOR-dyn-PT}
\LibT = L(0,T_i)
 \exp\left( \int_0^t b(s,T_i)\ud s+\int_0^t \volT\ud H_s \right),
\end{align}
where \prozess[H] is the $\P_{T_*}$-\tih \lev process driving the LIBOR rates,
cf. \eqref{canon-LIBOR}. The drift term $b(\cdot,\ts[i])$ has the form
\begin{align}\label{LIBOR-drift-PT}
b(s,\ts[i])
 & = -\half \vol2T c_s
     - c_s\volT \sum_{l=i+1}^{N}\frac{\delta_l L(s-,T_l)}{1+\delta_l L(s-,T_l)}\lambda(s,T_l)
     \nonumber\\&\quad
     - \int_{\R}\left(\Big(\e^{\volT x}-1\Big)\prod_{l=i+1}^{N}\beta(s,x,T_l)
                      -\volT x\right)F_s^{T_*}(\ud x),
\end{align}
where $\beta(s,x,T_l)$ is given by \eqref{Levy-LIBOR-beta}. Note that the drift
term of \eqref{LIBOR-dyn-PT} is random, therefore we are dealing with a general
semimartingale, and not with a \lev process. Of course, $L(\cdot,\ts[i])$ is not
a $\P_{T_*}$-martingale, unless $i=N$ (where we use the conventions $\sum_{l=1}^0=0$
and $\prod_{l=1}^0=1$).

Let us denote by $Z$ the log-\lib rates, that is
\begin{align}\label{log-LIB-SIE}
Z(t,T_i)
 &:= \log L(t,T_i) \nonumber\\
 &= Z(0,T_i) + \int_0^t b(s,T_i)\ud s + \int_0^t \volT\ud H_s,
\end{align}
where $Z(0,T_i)=\log L(0,T_i)$ for all $i\in\{1,\dots,N\}$. We can immediately
deduce that $Z(\cdot,T_i)$ is a semimartingale and its triplet of predictable
characteristics under $\pt[*]$, $\mathbb{T}(Z(\cdot,T_i)|\P_{T_*})=(B^i,C^i,\nu^i)$,
is described by
\begin{align}
 B^i &= \int\nolimits_0^\cdot b(s,\ts[i])\ds \nonumber\\
 C^i &= \int\nolimits_0^\cdot \vol2T c_s\ds\\
 1_A(x)*\nu^i &= 1_A\big(\volT x\big)*\nu^{T_*}, \qquad A\in\mathcal B(\Rmz). \nonumber
\end{align}
The assertion follows from the canonical decomposition of a semimartingale and
the triplet of characteristics of the stochastic integral process; see, for
example, Proposition 1.3 in \citeN{Papapantoleon06}.

Hence, the log-\lib rates satisfy the following linear SDE
\begin{align}\label{log-LIB-SDE}
\ud Z(t,T_i) &= b(t,T_i)\dt + \lambda(t,T_i)\ud H_t,
\end{align}
with initial condition $Z(0,T_i)=\log L(0,T_i)$.

\begin{remark}
Note that the martingale part of $Z(\cdot,T_i)$, i.e. the stochastic integral
$\int_0^\cdot\volT\ud H_s$,
is a \tih \lev process. However, the random drift term destroys the \lev property
of $Z(\cdot,T_i)$, as the increments are no longer independent.
\end{remark}

\section{Strong Taylor approximation and applications}
\label{main}

The aim of this section is to \emph{strongly approximate} the stochastic differential
equations for the dynamics of \lib rates under the terminal measure. This pathwise
approximation is based on the strong Taylor approximation of the random processes
$L(\cdot,T_l)$, $i+1\le l\le N$ in the drift $b(\cdot,\ts[i])$ of the semimartingale
driving the LIBOR rates $L(\cdot,\ts[i])$; cf. equations
\eqref{LIBOR-dyn-PT}--\eqref{log-LIB-SIE}. The idea behind the strong Taylor
approximation is the perturbation of the initial SDE by a real parameter and a
classical Taylor expansion around this parameter, with usual conditions for
convergence (cf. Definition \ref{taylor}).

\subsection{Definition}

We introduce a parameter $\eps\in\R$ and will approximate the terms
\begin{align}\label{to-approx}
L(t-,T_l)
\end{align}
which cause the drift term to be random, by their \emph{first-order strong
Taylor approximation}; cf. Lemma \ref{strong-Taylor}. Note that the map
$x\mapsto\frac{\delta_l x}{1+\delta_l x}$, appearing in the drift, is globally
Lipschitz with Lipschitz constant $\delta^*=\max_l \delta_l$.

The following definition of the strong Taylor approximation is taken by
\citeN{Siopacha06}; see also \citeN{SiopachaTeichmann07}. Consider a smooth
curve $\epsilon\mapsto W_\epsilon$, where $\epsilon\in\R$ and
$W_\epsilon\in L^2(\Omega;\R)$.

\begin{definition}\label{taylor}
A \emph{strong Taylor approximation} of order $n\geq0$ is a (truncated) power
series
\begin{align}
\mathbf{T}^n(W_\epsilon)
 := \sum_{k=0}^n \frac{\epsilon^k}{k!}\frac{\partial^k}{\partial\epsilon^k}\Big|_{\epsilon=0}W_\epsilon
\end{align}
such that
\begin{align}
\E\left[|W_\epsilon-\mathbf{T}^n(W_\epsilon)|\right]=o(\epsilon^n),
\end{align}
holds true as $\epsilon\rightarrow0$.
\end{definition}

Then, for Lipschitz functions $f:\R\to\Rp$ with Lipschitz constant $k$ we get
the following error estimate:
\begin{align}\label{err-est}
\E[|f(W_\epsilon)-f(\mathbf{T}^n(W_\epsilon))|]
 \le k \E[|W_\epsilon-\mathbf{T}^n(W_\epsilon)|]
 = k o(\epsilon^n).
\end{align}

\begin{remark}
It is important to point out that motivated by the idea of the Taylor series we
perform an expansion around $\eps=0$ and the estimate \eqref{err-est} is valid.
However, for the pathwise approximation of \lib rates we are interested in the
region $\eps\approx1$, and hope that the expansion yields adequate results; for
$\eps=0$ we would simply recover the ``frozen drift'' approximation. Numerical
experiments show that this approach indeed yields better results than the
``frozen drift'' approximation; cf. section \ref{numerics}.
\end{remark}

\subsection{Strong Taylor approximation}

In this section we develop a strong Taylor approximation scheme for the dynamics
of log-\lib rates.

Let us introduce the auxiliary process
$X^{\eps}(\cdot,\ts[i])=(X^{\eps}(t,\ts[i]))_{0\le t\le T_i}$ with initial values
$X^{\eps}(0,\ts[i])=L(0,\ts[i])$ for all $i\in\{0,\dots,N\}$
and all $\eps\in\R$. The dynamics of $X^{\eps}(\cdot,\ts[i])$ is described by perturbing
the SDE of the log-\lib rates by the perturbation parameter $\eps\in\R$:
\begin{align}\label{X-SDE}
\ud X^{\eps}(t,\ts[i])
 &= \eps \Big( b(t,\ts[i];X^\eps(t))\dt + \lambda(t,T_i)\ud H_t \Big),
\end{align}
where the drift term $b(\cdot,\ts[i];X^\eps(\cdot))$ is given by \eqref{LIBOR-drift-PT}.
The term $X^\eps(\cdot)$ in $b(\cdot,\ts[i];X^\eps(\cdot))$ emphasizes that the
drift term depends on all subsequent processes
$X^{\eps}(\cdot,\ts[i+1]),\dots,X^{\eps}(\cdot,\ts[N])$, which are also perturbed
by $\eps$. Note that for $\eps=1$ the processes $X^1(\cdot,T_i)$ and $Z(\cdot,T_i)$
are \textit{indistinguishable}.

\begin{remark}
In the sequel we will use the notation $\mathbf{T}$ as shorthand for
$\mathbf{T}^1$.
\end{remark}

\begin{lemma}\label{strong-Taylor}
The \emph{first-order strong Taylor approximation} of the random variable
$X^{\eps}(t,\ts[i])$ is given by:
\begin{equation}\label{TX}
\mathbf{T}\big( X^{\eps}(t,\ts[i])\big)
 = \log L(0,\ts[i])+\eps\frac{\partial}{\partial\eps}\big|_{\eps=0}X^{\eps}(t,\ts[i]),
\end{equation}
where the first variation process
$\frac{\partial}{\partial\eps}|_{\eps=0}X^{\eps}(\cdot,\ts[i])=:Y(\cdot,\ts[i])$
of $X^{\eps}(\cdot,\ts[i])$ is a \emph{\tih \lev process} with local characteristics
\begin{align}\label{FV-triplet}
b^{Y_i}_s &= b(s,\ts[i];X(0)) \nonumber\\
c^{Y_i}_s &= \lambda^2(s,T_i)c_s \\
\int 1_A(x)F^{Y_i}_s(\dx)
          &= \int 1_A\big(\lambda(s,T_i) x\big)F_s^{T_*}(\dx),
  \quad A\in\mathcal{B}(\R). \nonumber
\end{align}
\end{lemma}
\begin{proof}
By definition, the first-order strong Taylor approximation is given by the
truncated power series
\begin{equation}\label{TX-2}
\mathbf{T}\big( X^{\eps}(t,\ts[i])\big)
 = X^{0}(t,\ts[i])
 + \eps\frac{\partial}{\partial\eps}\big|_{\eps=0}X^{\eps}(t,\ts[i]).
\end{equation}
Since the curves $\eps\mapsto X^{\eps}(t,\ts[i])$ are smooth, and
$X^{\eps}(t,\ts[i])\in L^2(\Omega)$ by Assumption $(\mathbb{EM})$, we get that
strong Taylor approximations of arbitrary order can always be obtained, cf.
\citeN[Chapter 1]{KrieglMichor97}. In particular, for the first-order expansion
we have that
\begin{align}
\E\left[|X^{\eps}(t,\ts[i]) - \mathbf{T}\big( X^{\eps}(t,\ts[i])\big)|\right]
 = o(\epsilon).
\end{align}

The zero-order term of the Taylor expansion trivially satisfies
\begin{align*}
X^{0}(t,\ts[i])=X^0(0,\ts[i])
\text{ for all } t, \text{ since }
\ud X^{0}(t,\ts[i])=0.
\end{align*}
Of course, the initial values of the perturbed SDE coincide with the initial
values of the un-perturbed SDE, hence $X^0(0,T_i)=L(0,T_i)=:X(0,T_i)$.

The first variation process of $X^{\eps}(\cdot,\ts[i])$ with respect to $\eps$
is derived by differentiating \eqref{X-SDE}; hence, the dynamics is
\begin{align}\label{FV-SDE}
\ud\Big(\frac{\partial}{\partial\eps}\Big|_{\eps=0}X^{\eps}(t,\ts[i])\Big)
 &= b(t,\ts[i];X^\eps(t))|_{\eps=0}\dt + \lambda(t,T_i)\ud H_t \nonumber\\
 &= b(t,\ts[i];X(0))\dt + \lambda(t,T_i)\ud H_t.
\end{align}
We can immediately notice that in the drift term $b(\cdot,\ts[i];X(0))$ of the first
variation process, the random terms $X^{\eps}(t,\ts[i])$ are replaced by their
deterministic initial values $X(0,\ts[i])=Z(0,T_i)$.

Let us denote by $Y(\cdot,\ts[i])$ the first variation process of $X^{\eps}(\cdot,\ts[i])$.
The solution of the linear SDE \eqref{FV-SDE} describing the dynamics of the
first variation process yields
\begin{align}\label{FV-dyn}
Y(t,\ts[i])
 &= \int_0^t b(s,\ts[i];X(0))\ds + \int_0^t \lambda(s,T_i)\ud H_s.
 \end{align}
Since the drift term is deterministic and $H$ is a \tih \lev process we can
conclude that $Y(\cdot,\ts[i])$ is itself a \tih \lev process. The local characteristics
of $Y(\cdot,\ts[i])$ are described by \eqref{FV-triplet}.
\end{proof}

To summarize, by setting $\eps=1$ in Lemma \ref{strong-Taylor}, we have developed
the following approximation scheme for the logarithm of the random terms
$X^1(\cdot,T_i)=Z(\cdot,T_i)$ entering the drift:
\begin{align}\label{approx-log-LIBOR}
\mathbf{T}X(t,\ts[i])
 = \log L(0,\ts[i]) + \int_0^t b(s,\ts[i];X(0))\ds + \int_0^t \lambda(s,T_i)\ud H_s.
\end{align}
Comparing \eqref{approx-log-LIBOR} with \eqref{log-LIB-SIE} it becomes evident
that we are approximating the semimartingale $Z(\cdot,T_i)$ with the \tih \textit{\lev
process} $\mathbf{T}X(\cdot,\ts[i])$.

\begin{remark}
A consequence of this approximation scheme is that we can embed the ``frozen
drift'' approximation into our method. Indeed, the ``frozen drift'' approximation
is the \textit{zero-order} Taylor approximation, i.e.
$X^1(t,\ts[i]) \approx \log L(0,\ts[i])$.
The dynamics of LIBOR rates using this approximation will be denoted by
$\widehat{L}^0(\cdot,T_i)$.
\end{remark}

\subsection{Application to LIBOR models}

In this section, we will apply the strong Taylor approximation of the log-LIBOR
rates $Z(\cdot,T_i)$ by $\mathbf{T}X(\cdot,T_i)$ in order to derive a \emph{strong},
i.e. pathwise, approximation for the dynamics of log-\lib rates. That is, we
replace the random terms in the drift $b(\cdot,T_i;Z(\cdot))$ by the \lev process
$\mathbf{T}X(\cdot,T_i)$ instead of the semimartingale $Z(\cdot,T_i)$. Therefore,
the dynamics of the \textit{approximate} log-\lib rates are given by
\begin{align}
\widehat{Z}(t,T_i)
 &= Z(0,T_i) + \int_0^t b(s,\ts[i];\mathbf{T}X(s))\ds + \int_0^t \lambda(s,T_i)\ud H_s,
\end{align}
where the drift term is provided by
\begin{align}
b(s,\ts[i];\mathbf{T}X(s))
 & = -\half \vol2T c_s
     - c_s\volT \sum_{l=i+1}^{N}
          \frac{\delta_l \e^{\mathbf{T}X(s-,T_l)}}{1+\delta_l \e^{\mathbf{T}X(s-,T_l)}}\lambda(s,T_l)
     \nonumber\\&\,\,
     - \int_{\R}\left(\Big(\e^{\volT x}-1\Big)\prod_{l=i+1}^{N}\widehat\beta(s,x,T_l)
                      -\volT x\right)F_s^{T_*}(\ud x),
\end{align}
with
\begin{align}%\label{Levy-LIBOR-beta}
\widehat\beta(t,x,T_l,)
 = \frac{\delta_l \exp\big(\mathbf{T}X(t-,T_l)\big)}
        {1+\delta_l \exp\big(\mathbf{T}X(t-,T_l)\big)}\Big(\e^{\lambda(t,T_l)x}-1\Big) +1.
\end{align}

The main advantage of the strong Taylor approximation is that the resulting SDE
for $\widehat{Z}(\cdot,T_i)$ can be simulated more easily than the equation for
$Z(\cdot,T_i)$. Indeed, looking at \eqref{log-LIB-SDE} and \eqref{LIBOR-drift-PT}
again, we can observe that each LIBOR rate $L(\cdot,T_i)$ depends on all subsequent
rates $L(\cdot,T_l)$, $i+1\le l\le N$. Hence, in order to simulate $L(\cdot,T_i)$,
we should start by simulating the furthest rate in the tenor and proceed iteratively
from the end. On the contrary, the dynamics of $\widehat{Z}(\cdot,T_i)$ depend
only on the \lev processes $\mathbf{T}X(\cdot,T_l)$, $i+1\le l\le N$, which are
independent of each other. Hence, we can use \textit{parallel computing} to
simulate all approximate \lib rates simultaneously. This significantly increases
the speed of the Monte Carlo simulations while, as the numerical example reveals,
the empirical performance is very satisfactory.

\begin{remark}
Let us point out that this method can be applied to any LIBOR model driven by a
general semimartingale. Indeed, the properties of \lev processes are not essential
in the proof of Lemma \ref{strong-Taylor} or in the construction of the LIBOR
model. If we start with a \lib model driven by a general semimartingale, then
the structure of this semimartingale will be ``transferred'' to the first variation
process, and hence also to the dynamics of the strong Taylor approximation.
\end{remark}

\section{Numerical illustration}
\label{numerics}

The aim of this section is to demonstrate the accuracy and efficiency of the
Taylor approximation scheme for the valuation of options in the \lev \lib model
compared to the ``frozen drift'' approximation. We will consider the pricing of
caps and swaptions, although many other interest rate derivatives can be
considered in this framework.

We revisit the numerical example in \citeN[pp. 76-83]{Kluge05}. That is, we
consider a tenor structure $T_0=0, T_1=\frac12, T_2=1 \dots, T_{10}=5=T_*$,
constant volatilities
\begin{align*}
  \lambda(\cdot,T_1)= 0.20 \qquad
  \lambda(\cdot,T_2)= 0.19 \qquad
  \lambda(\cdot,T_3)= 0.18 \\
  \lambda(\cdot,T_4)= 0.17 \qquad
  \lambda(\cdot,T_5)= 0.16 \qquad
  \lambda(\cdot,T_6)= 0.15 \\
  \lambda(\cdot,T_7)= 0.14 \qquad
  \lambda(\cdot,T_8)= 0.13 \qquad
  \lambda(\cdot,T_9)= 0.12
\end{align*}
and the discount factors (zero coupon bond prices) as quoted on February 19,
2002; cf. Table \ref{tab:zcbprices}. The tenor length is constant and denoted
by $\delta=\frac12$.

{\renewcommand{\arraystretch}{1.16}
\begin{table}%[h]
 \begin{center}
  \begin{minipage}{11cm}\small
   \begin{tabular}{c|ccccc}
\hline
$T$      & 0.5\,Y&1\,Y&1.5\,Y&2\,Y&2.5\,Y\\
$B(0,T)$ & 0.9833630& 0.9647388& 0.9435826& 0.9228903& 0.9006922\\
\hline
$T$      &3\,Y&3.5\,Y&4\,Y&4.5\,Y&5\,Y\\
$B(0,T)$ & 0.8790279& 0.8568412& 0.8352144& 0.8133497& 0.7920573\\
 \hline
    \end{tabular}~\\[1ex]
   \end{minipage}
  \caption{Euro zero coupon bond prices on February 19, 2002.}
  \label{tab:zcbprices}
 \end{center}
\end{table}
}

The driving \lev process $H$ is a normal inverse Gaussian (NIG) process with
parameters $\alpha=\bar\delta=1.5$ and $\mu=\beta=0$. We denote by $\mu^H$ the
random measure of jumps of $H$ and by $\nu\dtdx=F(\dx)\dt$ the
$\P_{T_*}$-compensator of $\mu^H$, where $F$ is the \lev measure of the NIG
process. The necessary conditions are satisfied because $M=\alpha$, hence
$\sum_{i=1}^{9}|\lambda(\cdot,T_i)|=1.44<\alpha$
and $\lambda(\cdot,T_i)<\frac{\alpha}{2}$, for all $i\in\{1,\dots,9\}$.

The NIG \lev process is a pure-jump \lev process and, for $\mu=0$, has the
canonical decomposition
\begin{align}
 H=\int_0^\cdot\int_\R x (\mu^H-\nu)\dsdx.
\end{align}
The cumulant generating function of the NIG distribution is
\begin{align}
\kappa(u)
 &= \bar\delta\alpha-\bar\delta\sqrt{\alpha^2-u^2},
\end{align}
for all $u\in\mathbb{C}$ with $|\Re u|\le\alpha$.

\subsection{Caplets}

The price of a caplet with strike $K$ maturing at time $T_i$, using the
relationship between the terminal and the forward measures cf. \eqref{TF-libor},
can be expressed as
\begin{align}\label{caplet}
\mathbb{C}_0(K,\ts[i])
 &= \delta B(0,\ts[i+1])\, \E_{\pt[i+1]}[(L(T_i,T_i)-K)^+] \nonumber\\
 &= \delta B(0,\ts[i+1])\,
    \E_{\pt[*]}\Big[\frac{\ud\MeaT}{\ud\P_{T_*}}\big|_{\F_{\ts[i]}}(L(T_i,T_i)-K)^+\Big] \nonumber\\
 &= \delta B(0,\ts[*])\,
    \E_{\pt[*]}\Big[\prod_{l=i+1}^{N}\big(1+\delta L(T_i,T_l)\big)(L(T_i,T_i)-K)^+\Big].
\end{align}
This equation will provide the actual prices of caplets corresponding
to simulating the full SDE for the LIBOR rates. In order to calculate the
first-order Taylor approximation prices for a caplet we have to replace
$L(\cdot,T_\cdot)$ in \eqref{caplet} with $\widehat{L}(\cdot,\ts[\cdot])$.
Similarly, for the frozen drift approximation prices we must use
$\widehat L^0(\cdot,\ts[\cdot])$ instead of $L(\cdot,T_\cdot)$.

We will compare the performance of the strong Taylor approximation relative to
the frozen drift approximation in terms of their implied volatilities. In Figure
\ref{fig:caplets-IV-diffs} we present the difference in implied volatility between
the full SDE prices and the frozen drift prices, and between the full SDE prices
and the strong Taylor prices. One can immediately observe that the strong Taylor
approximation method performs
much better than the frozen drift approximation; the difference in implied
volatilities is very low across all strikes and maturities. Indeed, the
difference in implied volatility between the full SDE and the strong Taylor
prices lies always below the $1\%$ threshold, which deems this approximation
accurate enough for practical implementations. On the contrary, the difference
in implied volatilities for the frozen drift approximation exceeds the $1\%$
level for in-the-money options.

\begin{figure}
  \centering
  \includegraphics[width=6.25cm]{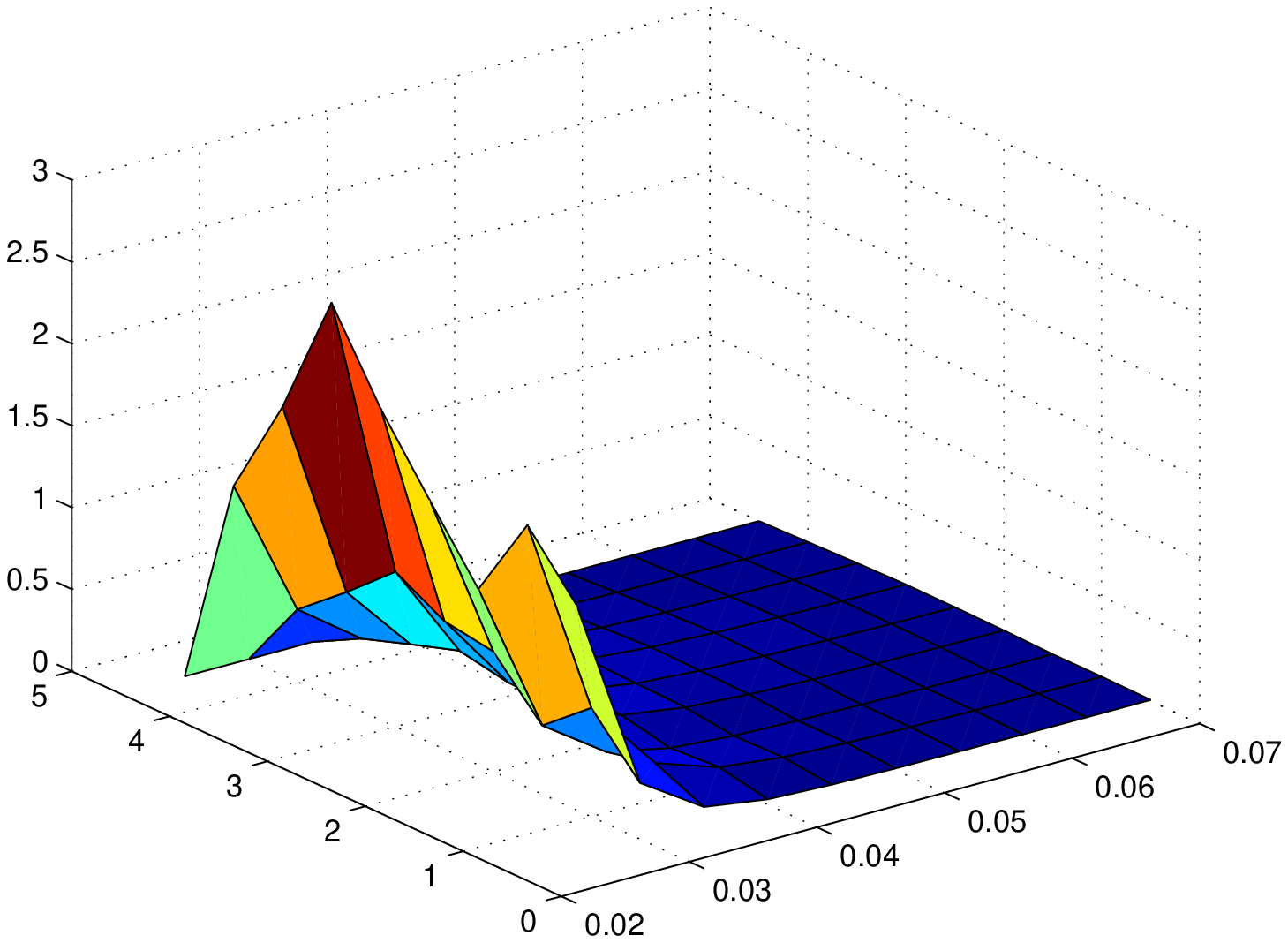}
  \includegraphics[width=6.25cm]{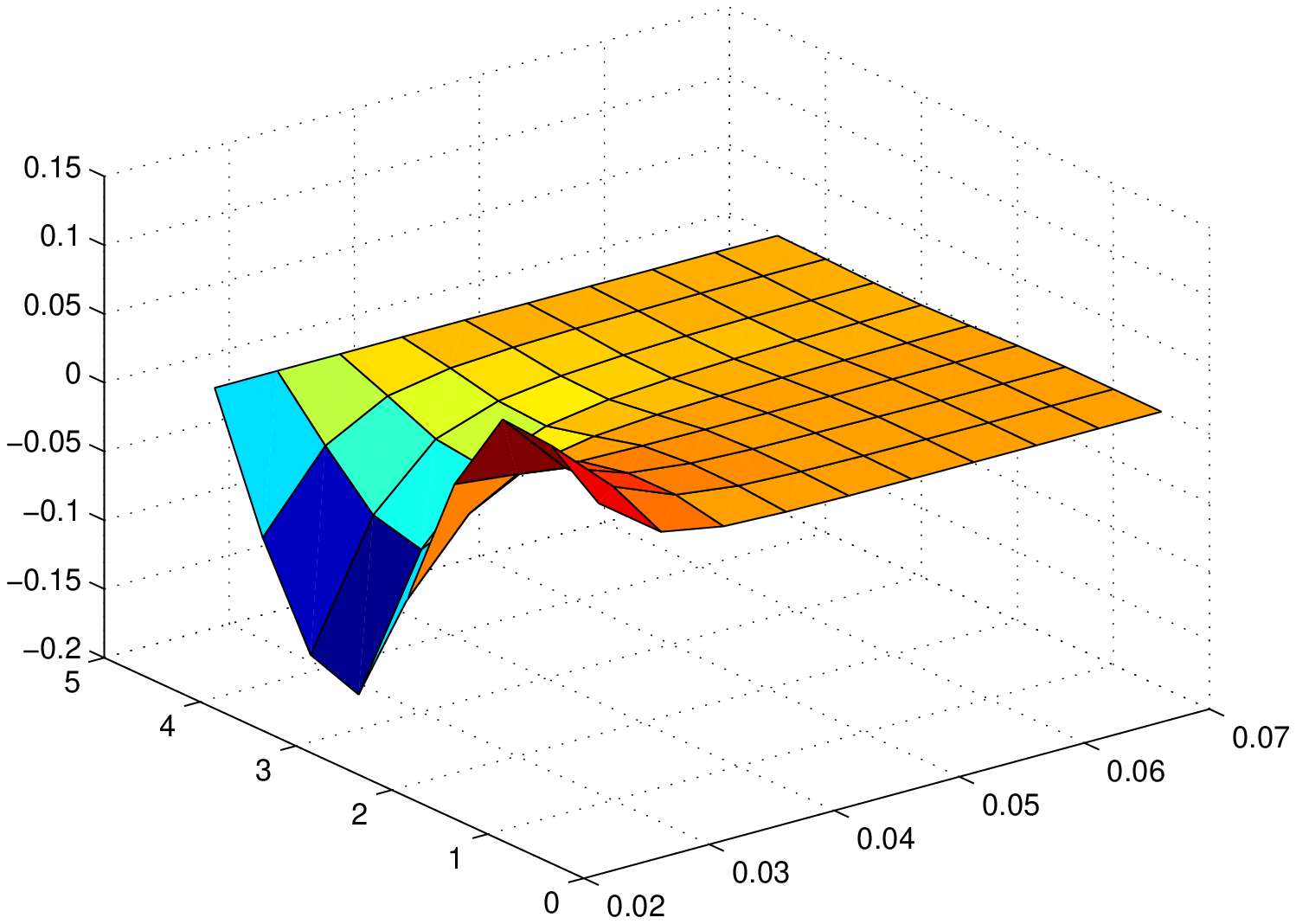}
  \caption{Difference in implied volatility between the full SDE and the frozen
           drift prices (left), and the full SDE and the strong Taylor prices (right).}
  \label{fig:caplets-IV-diffs}
\end{figure}

\subsection{Swaptions}

Next, we will consider the pricing of swaptions. Recall that a payer
(resp. receiver) swaption can be viewed as a put (resp. call) option
on a coupon bond with exercise price 1; cf. section 16.2.3 and 16.3.2
in \citeN{MusielaRutkowski97}. Consider a payer swaption with strike
rate $K$, where the underlying swap starts at time $T_i$ and matures
at $T_m$ ($i<m\le N$). The time-$T_i$ value is
\begin{align}
\mathbb{S}_{T_i}(K,T_i,T_m)
 &= \left( 1-\sum^m_{k=i+1} c_k B(T_i,T_k)\right)^+ \nonumber\\
 &= \left( 1-\sum^m_{k=i+1} \bigg(c_k \prod_{l=i}^{k-1}\frac{1}{1+\delta L(T_i,T_l)}\bigg)\right)^+,
\end{align}
where
\begin{align}%\label{}
c_k = \left\{
        \begin{array}{ll}
          K, & \hbox{$i+1\le k\le m-1$,} \\
          1+K, & \hbox{$k=m$.}
        \end{array}
      \right.
\end{align}
Then, the time-0 value of the swaption is obtained by taking the
$\pt[i]$-expectation of its time-$T_i$ value, that is
\begin{align}%\label{swap-1}
\mathbb{S}_{0}
 &= \mathbb{S}_{0}(K,T_i,T_m) \nonumber\\
 &= B(0,T_i)\,
    \E_{\pt[i]}\left[\left( 1-\sum^m_{k=i+1}
       \bigg(c_k \prod_{l=i}^{k-1}\frac{1}{1+\delta L(T_i,T_l)}\bigg)\right)^+\right] \nonumber\\
 &= B(0,T_*)\,\nonumber\\
 &\quad\times    \E_{\pt[*]}\left[\prod_{l=i}^{N}\big(1+\delta L(T_i,T_l)\big)
               \left( 1-\sum^m_{k=i+1}
               \bigg(c_k \prod_{l=i}^{k-1}\frac{1}{1+\delta L(T_i,T_l)}\bigg)\right)^+\right], \nonumber%\\
\end{align}
hence
\begin{align}\label{swap-1}
\mathbb{S}_{0} &= B(0,T_*)\,
    \E_{\pt[*]}\left[\left( - \sum^m_{k=i}
               \bigg(c_k \prod_{l=k}^{N}\left(1+\delta L(T_i,T_l)\right)\bigg)\right)^+\right],
\end{align}
where $c_i:=-1$. Once again, this equation will provide the actual prices of
swaptions corresponding to simulating the full SDE for the LIBOR rates. In order
to calculate the first-order Taylor approximation prices we have to replace $L(\cdot,T_\cdot)$
with $\widehat{L}(\cdot,\ts[\cdot])$, and for the frozen drift
approximation prices we must use
$\widehat L^0(\cdot,\ts[\cdot])$ instead of $L(\cdot,T_\cdot)$.

We will price eight swaptions in our tenor structure; we consider 1 year and 2
years as option maturities, and then use 12, 18, 24 and 30 months as swap
maturities for each option. Similarly to the simulations we performed for caplets,
we will simulate the prices of swaptions using all three methods and compare their
differences; these can be seen in Figure \ref{fig:swap-price-diff}.
Once again we observe that the strong Taylor method is performing
very well across all strikes, option maturities and swap maturities, while the
performance of the frozen drift method is poor for in-the-money swaptions and
seems to be deteriorating for longer swap maturities. This observation is in
accordance with the common knowledge that the frozen drift approximation is
performing worse and worse for longer maturities.

\begin{figure}
  \centering
  \includegraphics[width=6.25cm]{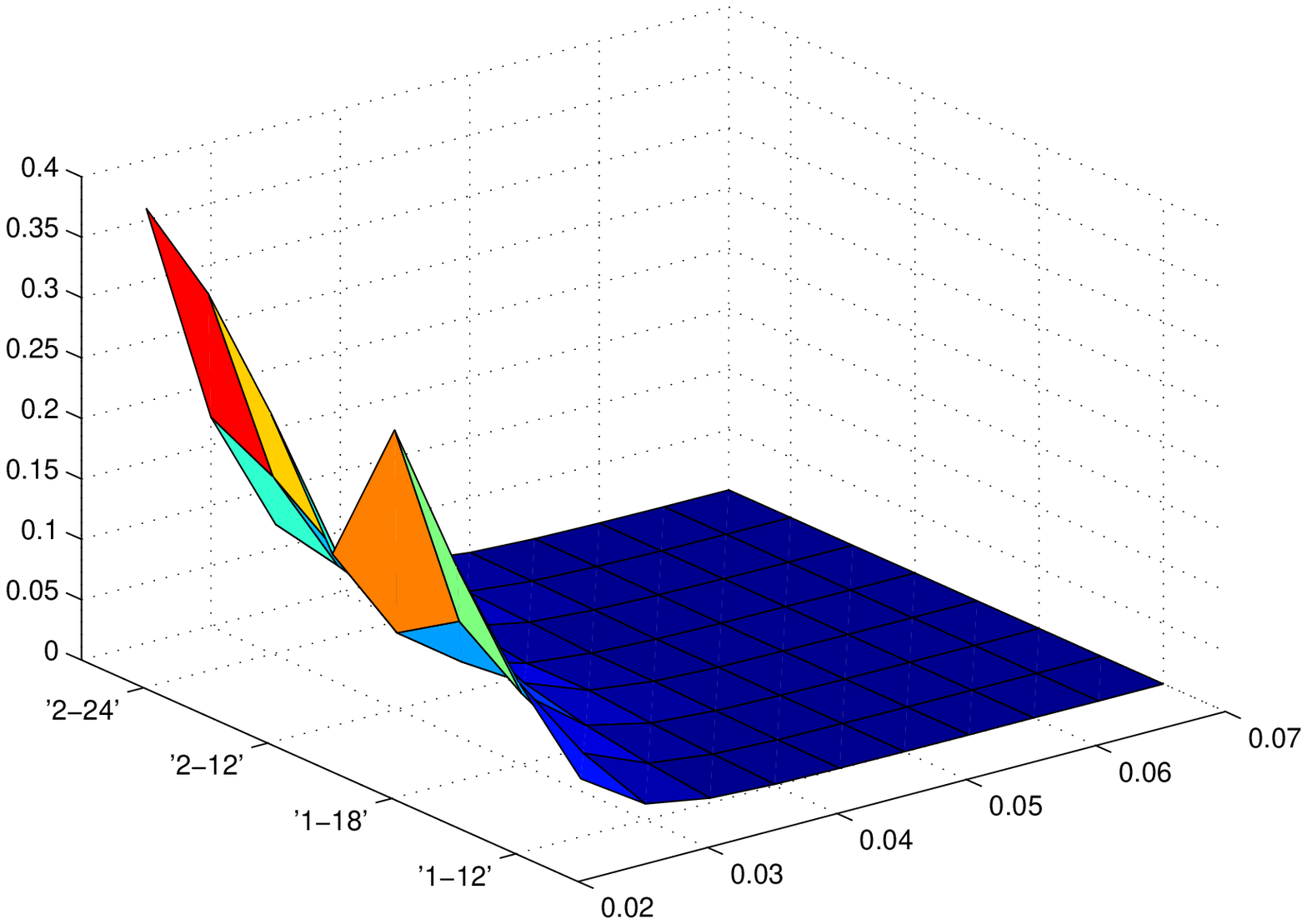}
  \includegraphics[width=6.25cm]{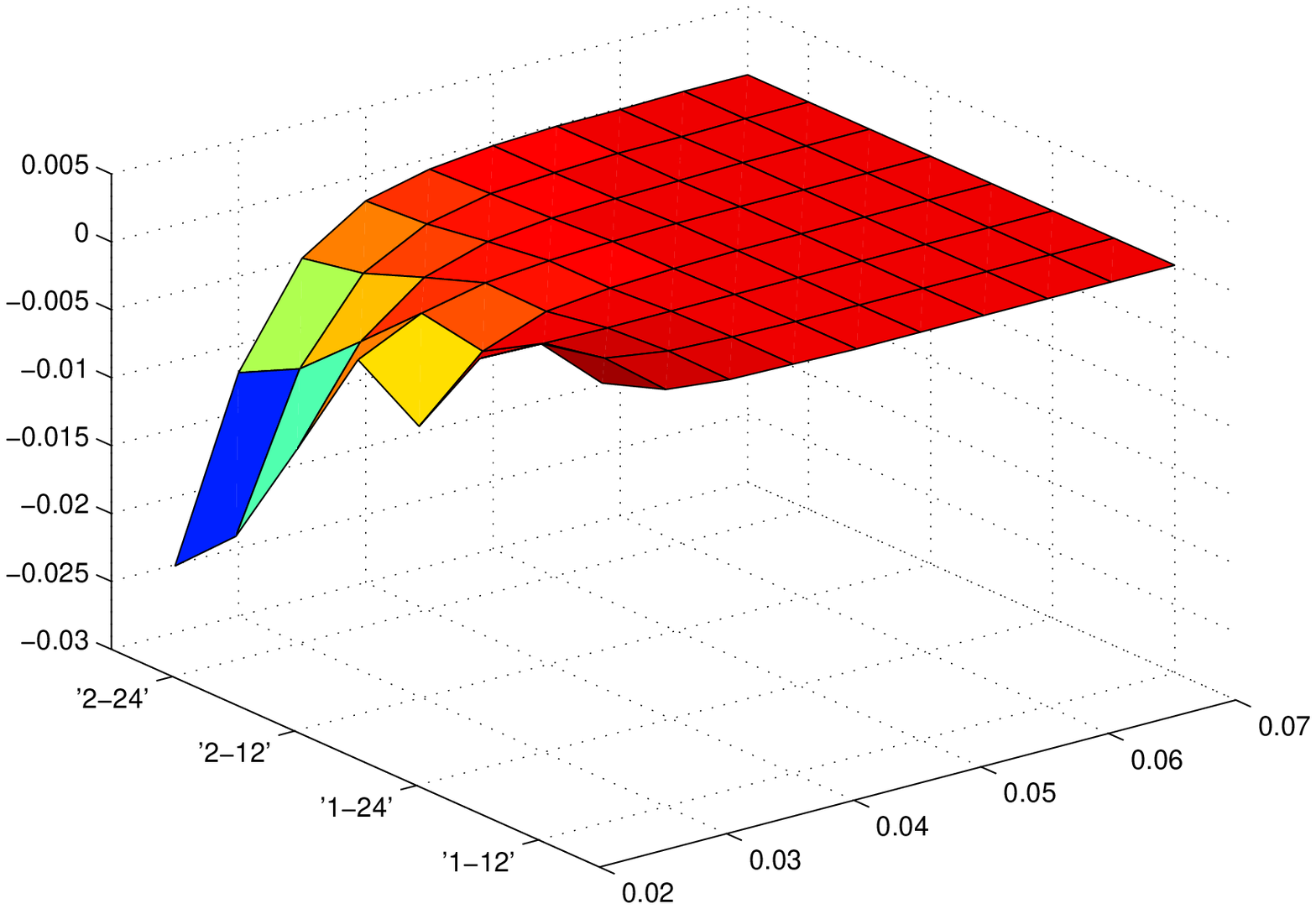}
  \caption{Difference in swaption prices between the full SDE and the frozen
           drift method (left), and the full SDE and the strong Taylor method (right).}
  \label{fig:swap-price-diff}
\end{figure}

\bibliographystyle{chicago}
\bibliography{references}

\begin{thebibliography}{}

\bibitem[\protect\citeauthoryear{Andersen and Brotherton-Ratcliffe}{Andersen
  and Brotherton-Ratcliffe}{2005}]{AndersenBrothertonRatcliffe05}
Andersen, L. and R.~Brotherton-Ratcliffe (2005).
\newblock Extended {LIBOR} market models with stochastic volatility.
\newblock {\em J. Comput. Finance\/}~{\em 9}, 1--40.

\bibitem[\protect\citeauthoryear{Black}{Black}{1976}]{Black76}
Black, F. (1976).
\newblock The pricing of commodity contracts.
\newblock {\em J. Financ. Econ.\/}~{\em 3}, 167--179.

\bibitem[\protect\citeauthoryear{Brace, Dun, and Barton}{Brace
  et~al.}{2001}]{BraceDunBarton01}
Brace, A., T.~Dun, and G.~Barton (2001).
\newblock Towards a central interest rate model.
\newblock In E.~Jouini, J.~Cvitani{\'c}, and M.~Musiela (Eds.), {\em Option
  pricing, interest rates and risk management}, pp.\  278--313. Cambridge
  University Press.

\bibitem[\protect\citeauthoryear{Brace, G\c{a}tarek, and Musiela}{Brace
  et~al.}{1997}]{BraceGatarekMusiela97}
Brace, A., D.~G\c{a}tarek, and M.~Musiela (1997).
\newblock The market model of interest rate dynamics.
\newblock {\em Math. Finance\/}~{\em 7}, 127--155.

\bibitem[\protect\citeauthoryear{Daniluk and G\c{a}tarek}{Daniluk and
  G\c{a}tarek}{2005}]{DanilukGatarek05}
Daniluk, A. and D.~G\c{a}tarek (2005).
\newblock A fully log-normal {LIBOR} market model.
\newblock {\em Risk\/}~{\em 18\/}(9), 115--118.

\bibitem[\protect\citeauthoryear{Dun, Barton, and Schl{\"o}gl}{Dun
  et~al.}{2001}]{DunBartonSchloegl01}
Dun, T., G.~Barton, and E.~Schl{\"o}gl (2001).
\newblock Simulated swaption delta-hedging in the lognormal forward {LIBOR}
  model.
\newblock {\em Int. J. Theor. Appl. Finance\/}~{\em 4}, 677--709.

\bibitem[\protect\citeauthoryear{Eberlein, Jacod, and Raible}{Eberlein
  et~al.}{2005}]{EberleinJacodRaible05}
Eberlein, E., J.~Jacod, and S.~Raible (2005).
\newblock {L\'{e}vy term structure models: no-arbitrage and completeness}.
\newblock {\em Finance Stoch.\/}~{\em 9}, 67--88.

\bibitem[\protect\citeauthoryear{Eberlein and \"Ozkan}{Eberlein and
  \"Ozkan}{2005}]{EberleinOezkan05}
Eberlein, E. and F.~\"Ozkan (2005).
\newblock {The L\'evy LIBOR model}.
\newblock {\em Finance Stoch.\/}~{\em 9}, 327--348.

\bibitem[\protect\citeauthoryear{Glasserman and Kou}{Glasserman and
  Kou}{2003}]{GlassermanKou03}
Glasserman, P. and S.~G. Kou (2003).
\newblock The term structure of simple forward rates with jump risk.
\newblock {\em Math. Finance\/}~{\em 13}, 383--410.

\bibitem[\protect\citeauthoryear{Glasserman and Merener}{Glasserman and
  Merener}{2003a}]{GlassermanMerener03b}
Glasserman, P. and N.~Merener (2003a).
\newblock {Cap and swaption approximations in LIBOR market models with jumps}.
\newblock {\em J. Comput. Finance\/}~{\em 7}, 1--36.

\bibitem[\protect\citeauthoryear{Glasserman and Merener}{Glasserman and
  Merener}{2003b}]{GlassermanMerener03}
Glasserman, P. and N.~Merener (2003b).
\newblock {Numerical solution of jump-diffusion LIBOR market models}.
\newblock {\em Finance Stoch.\/}~{\em 7}, 1--27.

\bibitem[\protect\citeauthoryear{Glasserman and Zhao}{Glasserman and
  Zhao}{2000}]{GlassermanZhao00}
Glasserman, P. and X.~Zhao (2000).
\newblock Arbitrage-free discretization of lognormal forward {LIBOR} and swap
  rate models.
\newblock {\em Finance Stoch.\/}~{\em 4}, 35--68.

\bibitem[\protect\citeauthoryear{Hunter, J\"ackel, and Joshi}{Hunter
  et~al.}{2001}]{HunterJaeckelJoshi01}
Hunter, C., P.~J\"ackel, and M.~Joshi (2001).
\newblock Getting the drift.
\newblock {\em Risk\/}~{\em 14}, 81--84.

\bibitem[\protect\citeauthoryear{Jacod and Shiryaev}{Jacod and
  Shiryaev}{2003}]{JacodShiryaev03}
Jacod, J. and A.~N. Shiryaev (2003).
\newblock {\em Limit Theorems for Stochastic Processes\/} (2nd ed.).
\newblock Springer.

\bibitem[\protect\citeauthoryear{Jamshidian}{Jamshidian}{1999}]{Jamshidian99}
Jamshidian, F. (1999).
\newblock {LIBOR} market model with semimartingales.
\newblock Working Paper, NetAnalytic Ltd.

\bibitem[\protect\citeauthoryear{Joshi and Stacey}{Joshi and
  Stacey}{2008}]{JoshiStacey08}
Joshi, M. and A.~Stacey (2008).
\newblock New and robust drift approximations for the {LIBOR} market model.
\newblock {\em Quant. Finance\/}~{\em 8}, 427--434.

\bibitem[\protect\citeauthoryear{Kluge}{Kluge}{2005}]{Kluge05}
Kluge, W. (2005).
\newblock {\em {Time-inhomogeneous L\'evy processes in interest rate and credit
  risk models}}.
\newblock Ph.\ D. thesis, Univ. Freiburg.

\bibitem[\protect\citeauthoryear{Kriegl and Michor}{Kriegl and
  Michor}{1997}]{KrieglMichor97}
Kriegl, A. and P.~W. Michor (1997).
\newblock {\em The Convenient Setting of Global Analysis}.
\newblock American Mathematical Society.

\bibitem[\protect\citeauthoryear{Kurbanmuradov, Sabelfeld, and
  Schoenmakers}{Kurbanmuradov
  et~al.}{2002}]{KurbanmuradovSabelfeldSchoenmakers}
Kurbanmuradov, O., K.~Sabelfeld, and J.~Schoenmakers (2002).
\newblock Lognormal approximations to {LIBOR} market models.
\newblock {\em J. Comput. Finance\/}~{\em 6}, 69--100.

\bibitem[\protect\citeauthoryear{Miltersen, Sandmann, and Sondermann}{Miltersen
  et~al.}{1997}]{MiltersenSandmannSondermann97}
Miltersen, K.~R., K.~Sandmann, and D.~Sondermann (1997).
\newblock Closed form solutions for term structure derivatives with log-normal
  interest rates.
\newblock {\em J. Finance\/}~{\em 52}, 409--430.

\bibitem[\protect\citeauthoryear{Musiela and Rutkowski}{Musiela and
  Rutkowski}{1997}]{MusielaRutkowski97}
Musiela, M. and M.~Rutkowski (1997).
\newblock {\em Martingale Methods in Financial Modelling}.
\newblock Springer.

\bibitem[\protect\citeauthoryear{Papapantoleon}{Papapantoleon}{2007}]{Papapant%
oleon06}
Papapantoleon, A. (2007).
\newblock {\em Applications of semimartingales and {L}\'evy processes in
  finance: duality and valuation}.
\newblock Ph.\ D. thesis, Univ. Freiburg.

\bibitem[\protect\citeauthoryear{Pelsser, Pietersz, and van
  Regenmortel}{Pelsser et~al.}{2005}]{PelsserPieterszRegenmortel}
Pelsser, A., R.~Pietersz, and M.~van Regenmortel (2005).
\newblock Bridging {B}rownian {LIBOR}.
\newblock {\em Wilmott Mag.\/}~{\em 18}, 98--103.

\bibitem[\protect\citeauthoryear{Sandmann, Sondermann, and Miltersen}{Sandmann
  et~al.}{1995}]{SandmannSondermannMiltersen95}
Sandmann, K., D.~Sondermann, and K.~R. Miltersen (1995).
\newblock Closed form term structure derivatives in a {Heath--Jarrow--Morton}
  model with log-normal annually compounded interest rates.
\newblock In {\em Proceedings of the Seventh Annual European Futures Research
  Symposium Bonn}, pp.\  145--165.
\newblock Chicago Board of Trade.

\bibitem[\protect\citeauthoryear{Schl\"ogl}{Schl\"ogl}{2002}]{Schloegl02}
Schl\"ogl, E. (2002).
\newblock A multicurrency extension of the lognormal interest rate market
  models.
\newblock {\em Finance Stoch.\/}~{\em 6}, 173--196.

\bibitem[\protect\citeauthoryear{Schoenmakers}{Schoenmakers}{2005}]{Schoenmake%
rs05}
Schoenmakers, J. (2005).
\newblock {\em Robust {LIBOR} Modelling and Pricing of Derivative Products}.
\newblock Chapman \& Hall/CRC.

\bibitem[\protect\citeauthoryear{Siopacha}{Siopacha}{2006}]{Siopacha06}
Siopacha, M. (2006).
\newblock {\em {Taylor expansions of option prices by means of Malliavin
  calculus}}.
\newblock Ph.\ D. thesis, Vienna University of Technology.

\bibitem[\protect\citeauthoryear{Siopacha and Teichmann}{Siopacha and
  Teichmann}{2010}]{SiopachaTeichmann07}
Siopacha, M. and J.~Teichmann (2010).
\newblock {Weak and strong Taylor methods for numerical solutions of stochastic
  differential equations}.
\newblock {\em Quant. Finance\/}.
\newblock (forthcoming).

\end{thebibliography}

\end{document}